\documentclass[leqno,12pt]{amsart} 
\usepackage{amssymb}
\usepackage{amssymb,amsmath,amsfonts,amscd}
\usepackage{psfrag,graphicx}
\usepackage[all,dvips,cmtip]{xy}
\usepackage{enumerate}
\newenvironment{enumi}{\begin{enumerate}[\upshape \hspace{0.25cm} (i)]}{\end{enumerate}}
\newenvironment{enuma}{\begin{enumerate}[\upshape \hspace{0.5cm} (a)]}{\end{enumerate}}
\newcommand{\IZ}{\mathbb{Z}}
\newcommand{\IQ}{\mathbb{Q}}
\newcommand{\IC}{\mathbb{C}}
\newcommand{\IA}{\mathbb{A}}
\newcommand{\IP}{\mathbb{P}}
\newcommand{\uIC}{\ul{\mathbb{C}}}

\newcommand{\cC}{\mathcal C}
\newcommand{\cG}{\mathcal G}
\newcommand{\cL}{\mathcal L}
\newcommand{\cM}{\mathcal M}
\newcommand{\cO}{\mathcal O}
\newcommand{\cR}{\mathcal R}
\newcommand{\cS}{\mathcal S}
\newcommand{\cT}{\mathcal T}

\newcommand{\cX}{\mathcal X}
\newcommand{\cY}{\mathcal Y}

\newcommand{\mfm}{\mathfrak{m}}
\newcommand{\mfa}{\mathfrak{a}}

\newcommand{\al}{\alpha}
\newcommand{\be}{\beta}
\newcommand{\cd}{\cdot}

\newcommand{\De}{\Delta}
\newcommand{\de}{\delta}
\newcommand{\bD}{\mathbf{\Delta}}
\newcommand{\e}{\epsilon}

\newcommand{\ga}{\gamma}

\newcommand{\lam}{\lambda}
\newcommand{\lan}{\langle}

\newcommand{\mcal}{\mathcal}

\newcommand{\mb}{\mbox}

\newcommand{\nf}{\normalfont}

\newcommand{\op}{\oplus}
\newcommand{\ot}{\otimes}

\newcommand{\ran}{\rangle}
\newcommand{\ra}{\rightarrow}
\newcommand{\si}{\sigma}
\newcommand{\Si}{\Sigma}
\newcommand{\bS}{\mathbf{\Sigma}}
\newcommand{\ti}{\tilde}
\newcommand{\ul}{\underline}

\theoremstyle{plain} 
\newtheorem{theorem}{\indent\sc Theorem}[section]
\newtheorem{lemma}[theorem]{\indent\sc Lemma}
\newtheorem{corollary}[theorem]{\indent\sc Corollary}
\newtheorem{proposition}[theorem]{\indent\sc Proposition}

\theoremstyle{definition} 
\newtheorem{definition}[theorem]{\indent\sc Definition}
\newtheorem{remark}[theorem]{\indent\sc Remark}
\newtheorem{example}[theorem]{\indent\sc Example}
\newtheorem{notation}[theorem]{\indent\sc Notation}

\begin{document}

\title[On toric Deligne-Mumford stacks]{A note on toric Deligne-Mumford stacks} 

\author[F. Perroni]{Fabio Perroni}


\thanks{This research was partially supported by SNF, No 200020-107464/1.
}
\address{
Institut f\"ur Mathematik \endgraf
Universit\"at Z\"urich \endgraf
Winterthurerstrasse 190 \endgraf
8057 Z\"urich \endgraf
Switzerland
}
\email{fabio.perroni@math.unizh.ch}


\maketitle

\begin{abstract}
We give a new description of the data needed to specify a morphism 
from a scheme to a toric Deligne-Mumford stack. The description
is given in terms of a collection of line bundles and sections
which satisfy certain conditions.
As applications, we
characterize any toric Deligne-Mumford stack as a product of roots
of line bundles over the rigidified stack, 
describe the torus action, 
describe morphisms between toric Deligne-Mumford stacks with
complete coarse moduli spaces in terms of homogeneous polynomials,
and compare two different definitions of toric stacks.
\end{abstract}

\section{Introduction} 
A map from a scheme $Y$ to the projective space $\IP^d$ is determined by a line bundle $L$ on $Y$
together with $d+1$ sections which do not vanish simultaneously.
More generally,  when $X$ is a smooth toric variety
a map $Y\ra X$ is determined by a collection of line bundles and sections
on $Y$ which satisfy certain compatibility and nondegeneracy conditions \cite{cox}.
An analogous result holds in the case of a  simplicial 
toric variety $X$: there is a natural orbifold structure $\cX$ on $X$
such that a map $Y\ra \cX$ is determined by a collection of line bundles and sections
on $Y$ precisely as in the case of $X$ smooth \cite{DH}. More recently,
toric Deligne-Mumford stacks have been defined \cite{BCS}.
They are smooth separated Deligne-Mumford stacks whose coarse
moduli spaces are simplicial toric varieties and such that the
automorphism group of the generic point is not necessarily trivial. 

The goal of the present paper is to generalize the results in \cite{cox}
and \cite{DH} in order to describe morphisms $Y\ra \cX$
where $\cX$ is a toric Deligne-Mumford stack in the sense of \cite{BCS} (Theorem \ref{mainth}).
As an application of this result we deduce some geometric
properties of toric Deligne-Mumford stacks. We  show that, given $\cX$,
the rigidification with respect to the generic automorphism group
$\cX \ra \cX_{\rm rig}$ is isomorphic to the fibered product of 
roots of certain line bundles over $\cX_{\rm rig}$ (Proposition \ref{ger}). 
This result is  used to obtain a classification
of toric Deligne-Mumford stacks in terms of the combinatorial data $\bD$ 
(Theorem \ref{classificazione}).
In Section 4 we describe the torus action. In Section 5
we show how homogeneous polynomials can be used
to describe all maps $\cY \ra \cX$
between the toric Deligne-Mumford stacks $\cY$ and $\cX$ whose
coarse moduli spaces are complete varieties (Theorem \ref{mappe polinomiali}).
Finally, we compare two different definitions of toric stacks:
the one of \cite{BCS} and that of \cite{Iwanari}.

The construction presented in \cite{BCS} is inspired
by the quotient construction for toric varieties.
On the other hand, a toric variety $X$ is a separated, normal variety
containing a torus as dense open subvariety such that the multiplication
of the torus extends to an action on $X$. This point of view inspired
the work \cite{FMN}, where the definition of 
``smooth toric Deligne-Mumford stack'' is given.
Toric Deligne-Mumford stacks are smooth stacks and they are 
``smooth'' in the sense of \cite{FMN}.
Throughout the paper and in the appendix we will make some remarks 
on the relations between the two constructions.

\medskip

\noindent \textit{Notation}.  We denote by $\mathbf{\De}$ the following set of data:
\begin{enumi}
\item a free abelian group $N$ of rank $d$,
\item a simplicial fan $\De$ in $N_{\IQ}$ in the sense of \cite{F},
where $N_{\IQ}:= N\ot_{\IZ} \IQ$,
\item an element $a_{\rho} \in \rho \cap N$ for any $\rho \in \De(1)$,
where $\De(1)$ is the set of $1$-dimensional cones of $\De$,
\item a sequence of $R$ positive non zero integers $r_1,...,r_R$,
\item and an integer $b_{i \rho} \in \IZ$, for any $i\in \{1,...,R\}$ and
             $\rho \in  \De(1) $.
\end{enumi}
We denote with $M$ the dual lattice of $N$, $M:= {\rm Hom}_{\IZ}(N,\IZ)$, and
with $\De_{\rm max}$  the set of maximal cones in $\De$
with respect to the inclusion.
We will use $\ot_{\rho}$ to denote $\ot_{\rho \in \De(1)}$,
$\ot_{i}$ to denote $\ot_{i =1}^R$, and similarly for $\sum_{\rho}$
and $\sum_{i}$.

We  work over the field of complex numbers $\IC$.
The site given by the category of $\IC$-schemes with
the \'etale topology is denoted by (Sch).
We write DM-stack instead of Deligne-Mumford stack.
For any scheme $Y$, we denote by $\ul{\IC}$ the trivial line bundle
over it.

\medskip

\noindent \textit{Acknowledgments}. This work began after a series of discussions with  Christian Okonek,
I am  grateful to him.
I am grateful to Markus D\"urr and Mihai Halic for bring at my attention their
paper and for further discussions.
I wish to thank Markus Bader,   Barbara Fantechi,
Andrew Kresch, \'Etienne Mann, Fabio Nironi, Markus Perling, 
and Jonathan Wise for useful discussions;
and moreover Isamu Iwanari and Martin Olsson for suggesting interesting 
connections with logarithmic geometry.
Part of the work was carried out
during a visit at the Institut Henri Poincar\'e, in occasion of the Program
"Groupoids and Stacks in Physics and Geometry", I thank  the organizers
for inviting me  and for the kind hospitality.

\section{$\mathbf{\De}$-collections.}
$\De$-collections were introduced in \cite{cox}
 as the data needed to specify a morphism
from a scheme to a smooth toric variety (see Example \ref{mce} below).
The generalization to the case of a simplicial fan $\De$ has been
studied in \cite{DH}:
when $R=0$ the content of Theorem \ref{mainth} below coincides with that of 
Theorem 16 in \cite{DH}. In this Section we investigate the case
of morphisms 
to a toric DM-stack (in the sense of \cite{BCS}). We define $\bD$-collections
and morphisms between them
as collections of line bundles and sections over a given scheme
which satisfy certain compatibility and nondegeneracy conditions.
The main result is Theorem \ref{mainth} where we prove that the category of 
$\bD$-collections $\cC_{\bD}$ is a separated smooth DM-stack and,
if the $1$-dimensional cones of $\De$ span $N_{\IQ}$, is isomorphic to the
toric DM-stack $\cX_{\bS}$, where $\bS$ is a stacky fan
determined by $\bD$. 

\begin{definition}\label{deltacoll}
Let $Y$ be a scheme. A \textit{$\bD$-collection} on $Y$ is the data of
\begin{enumi}
\item a line bundle $L_{\rho}$ on $Y$ and a section  $u_{\rho}\in H^0(Y, L_{\rho})$
for any $\rho \in \De (1)$,
\item  isomorphisms $c_m : \ot_{\rho} L_{\rho}^{\ot \lan m, a_{\rho}\ran} \ra \uIC$ for any $\quad m\in M$,
\item  a line bundle $M_i$ on $Y$ and an isomorphism
$d_i : \ot_{\rho}L_{\rho}^{\ot b_{i \rho}}\ot M_i^{\ot r_i} \ra \uIC$ for any $i\in \{ 1,...,R \}$,
\end{enumi}
such that the following conditions are satisfied:
\begin{enumerate}
\item $c_m \ot c_{m'} = c_{m+m'}$ for any $m,m' \in M$,
\item for any $y\in Y$, there exists a cone $\si \in \De_{\rm max}$
such that $u_{\rho}(y)\not= 0$ for all $\rho \not\in \si$.
\end{enumerate}

A $\bD$-collection on $Y$ is written $(L_{\rho}, u_{\rho}, c_m, M_i, d_i)/Y$.
\end{definition}

\begin{definition}\label{deltacoll2}
Let  $(L_{\rho}', u_{\rho}', c_m', M_i', d_i')/Y'$
and $(L_{\rho}, u_{\rho}, c_m, M_i, d_i)/Y$ be two $\bD$-collections.
A \textit{morphism} from $(L_{\rho}', u_{\rho}', c_m', M_i', d_i')/Y'$
to $(L_{\rho}, u_{\rho}, c_m, M_i, d_i)/Y$ is given by a morphism
$f:Y' \ra Y$ of schemes, morphisms $\ga_{\rho}:L_{\rho}' \ra  L_{\rho} $
and $\de_i:M_i' \ra  M_i$ of line bundles for any $\rho \in \De (1)$
and any $i\in \{ 1,...,R \}$, such that  the following conditions are
satisfied:
\begin{itemize}
\item the $\ga_{\rho}$'s induce  isomorphisms $L_{\rho}' \ra
  f^*L_{\rho}$ and  the same holds for the $\de_i$'s;
\item $\ga_{\rho}\circ u_{\rho}'=  u_{\rho}\circ f$ for all $\rho\in
  \De (1)$;
\item for any $m\in M$ the following diagram commutes
\begin{equation*}
\begin{CD}
\ot_{\rho}  L_{\rho}'^{\ot \lan m, a_{\rho}\ran} @>c_m' >> \uIC \\
@V \ot_{\rho}\ga_{\rho}^{\ot \lan m, a_{\rho}\ran} VV @VV f \times
{\rm id}_{\IC} V \\
\ot_{\rho}  L_{\rho}^{\ot \lan m, a_{\rho}\ran} @>c_m>>  \uIC
\end{CD}
\end{equation*}
\item for any $i\in \{1,...,R\}$ the following diagram  commutes
\begin{equation*}
\begin{CD}
\ot_{\rho}L_{\rho}'^{\ot b_{i \rho}}\ot M_i'^{\ot r_i} @>d_i' >> \uIC \\
@V \ot_{\rho}\ga_{\rho}^{\ot b_{i \rho}}\ot {\de}_i^{\ot r_i} VV @VV f \times
{\rm id}_{\IC} V \\
\ot_{\rho} L_{\rho}^{\ot b_{i \rho}}\ot M_i^{\ot r_i} @> d_i>> \uIC.
\end{CD}
\end{equation*}
\end{itemize}

A morphism from $(L_{\rho}', u_{\rho}', c_m', M_i', d_i')/Y'$ to
$(L_{\rho}, u_{\rho}, c_m, M_i, d_i)/Y$ is denoted by $(f,\ga_{\rho}, \de_{i})$.
\end{definition}

\vspace{0.5cm}

Let us consider some   examples.
\begin{example}\label{mce} \nf Let $\De$ be a fan and $N$ be a lattice 
which determine a smooth toric 
variety $X$. Set the $a_{\rho}$s the minimal lattice points
of the rays,  $R=0$. In this case $\bD$ is determined by $\De$ and
$N$, so we talk about $\De$-collections.
On $X$ there is a canonical $\De$-collection
defined as follows. For any $\rho \in \De(1)$, 
let  $D_{\rho}$ be the corresponding effective
Weil divisor stable under the torus action. Since $X$ is smooth, 
it corresponds to an effective Cartier divisor,
hence it  gives a line bundle $L_{\rho}$
with a section $u_{\rho}\in H^0(X,L_{\rho})$.
For any $m\in M$, the character $\chi^m$ is a rational function on $X$
such that ${\rm div}(\chi^m)= \sum_{\rho}\lan m,a_{\rho}\ran D_{\rho}$,
hence we get an isomorphism
$$
c_{\chi^m} : \ot_{\rho} L_{\rho}^{\ot \lan m,a_{\rho}\ran} \ra \uIC.
$$
Then, $(L_{\rho},u_{\rho},c_{\chi^m})$ is a $\De$-collection
on $X$ (Lemma 1.1 in \cite{cox}). This $\De$-collection is called \textit{universal} 
because of the following  result.
Let 
$$
C_{\De}: ({\rm Sch}) \ra ({\rm Sets})
$$
be the contravariant functor that associates to any scheme $Y$ the set
of equivalence classes of $\De$-collections on $Y$. Then $X$ is the fine 
moduli space for $C_{\De}$, and $(L_{\rho},u_{\rho},c_{\chi^m})$
is the universal family \cite{cox}.
\end{example}

\begin{example}\label{ex1.4}\nf Let $N:= \IZ$, $\De := \{ \{ 0\}, \rho:= \IQ_{\geq 0}\}$, $a_{\rho}:= a$, and
$R:=0$. A $\bD$-collection is given by a line bundle $L$ on $Y$, a section
$u\in H^0(Y,L)$, and an isomorphism $c: L^{\ot a} \ra \uIC$.
Then the category of $\bD$-collections is equivalent to the stack 
$[\IA^1_{\IC}/\mu_a]$, where $\mu_a$ is the group of $a$-th roots of the unity
acting by multiplication on $\IA^1_{\IC}$.
\end{example}
The next example is the analog of Example 3.5 in \cite{BCS}.
\begin{example}
Let $N:=\IZ$, $\De:= \{ \{0\}, \rho:= \IQ_{\leq 0}, \tau :=\IQ_{\geq 0} \}$,
$a_{\rho}:= -3$, $a_{\tau}:= 2$, $r=2$, $b_{\rho}:=0$ and $b_{\tau}:=1$. \\
A $\bD$-collection over $Y$ is given by 
\begin{itemize}
\item line bundles $L_{\rho}$ and $L_{\tau}$ over $Y$ with  sections $u_{\rho}\in H^0(Y,L_{\rho})$ and $u_{\tau}\in H^0(Y,L_{\tau})$
which do not vanish simultaneously,
\item an isomorphism $c: (L_{\rho}^{\vee})^{\ot 3}\ot L_{\tau}^{\ot 2} \ra \ul{\IC}$,
\item a line bundle $M$ over $Y$ and an isomorphism $d: L_{\tau} \ot M^{\ot 2} \ra \ul{\IC}$.
\end{itemize}

Let $\IP(3,2)$ be the weighted projective line $[\IC^2-\{ 0 \} / \IC^*]$, where 
$\IC^*$ acts with weights $2,3$. Let $\cO (1)$ be the line bundle on  $\IP(3,2)$
associated to the character $\rm{id}_{\IC^*}$ as in \cite{BH}. 
We define a functor from the category of $\bD$-collections
to $\sqrt{\cO (1)/\IP(3,2)}$ (in Section 3 we will review 
the definition of roots of line bundles).\\
For any $\bD$-collection $(L_{\rho},L_{\tau},u_{\rho},u_{\tau}, M, d)/Y$,
we define a morphism $Y \ra \sqrt{\cO (1)/\IP(3,2)}$ as follows:
the morphism $Y\ra \IP(3,2)$ is determined by the line bundle $L:= L_{\rho}^{\vee}\ot L_{\tau}$
and the sections $u_{\rho} \in H^0(Y,L^{\ot 2})$ and  $u_{\tau} \in H^0(Y,L^{\ot 3})$,
where we have used $c$ to identify $L^{\ot 2}$ with $L_{\rho}$
and $L^{\ot 3}$ with $L_{\tau}$;
 the square root of $\cO (1)$ is $(M\ot (L_{\rho}^{\vee}\ot L_{\tau}),d)$. 
The correspondence between arrows is defined in an analogous way.
Then the resulting functor is an equivalence of categories. 
\end{example}

\vspace{0.5cm}

The main result of the present paper (Theorem \ref{mainth}) states that, in general, 
there is a correspondence
between combinatorial data $\bD$ and toric DM-stacks.
Let $\mcal{C}_{\bD}$ be the category whose objects are
$\bD$-collections and morphisms are morphisms between $\bD$-collections.
The functor
$$
p: \mcal{C}_{\bD} \ra \rm{(Sch)}
$$
which sends the $\bD$-collection $(L_{\rho}, u_{\rho}, c_m, M_i, d_i)/Y$
to $Y$ and the morphism $(f,\ga_{\rho}, \de_{i})$ to $f$ makes
$\mcal{C}_{\bD}$ a category fibered in groupoids (a CFG) over the site (Sch).

\begin{theorem}\label{mainth}
The category fibered in groupoids $p:\mcal{C}_{\bD} \ra \rm{\nf{(Sch)}}$
is a separated smooth Deligne-Mumford stack whose
coarse moduli space is the toric variety  associated to
the fan $\De$ and the lattice $N$.

Moreover, if the $1$-dimensional cones of $\De$ span $N_{\IQ}$,
set\\
$$
\bS := ( N \op_{i=1}^R \IZ/ r_i\, ,\, \De \, ,\, \{ (a_{\rho},
\overline{b_{1 \rho}},...,\overline{b_{R \rho}})   \}_{\rho} \,),
$$
\\
where $\overline{b_{i \rho}}$ denotes the class of $b_{i \rho}$ in $ \IZ/ r_i$.
Then $\mcal{C}_{\bD}$ is isomorphic to the toric Deligne-Mumford 
stack $\mcal{X}_{\bS}$ associated to the stacky fan $\bS$ 
as defined in \cite{BCS}.
\end{theorem}

\vspace{0.2cm}

We notice that one can define $\mathbf{\De}$-collections over a stack
formally in the same way as for schemes.
Moreover, on $\mcal{C}_{\mathbf{\De}}$ there is a tautological  $\mathbf{\De}$-collection:
$$
(\mathfrak{L}_{\rho}, \mathfrak{u}_{\rho}, \mathfrak{c}_{m}, \mathfrak{M}_i, \mathfrak{d}_i)/\mcal{C}_{\mathbf{\De}}.
$$
As a result the theory of descent gives the following
\begin{corollary}\label{mfs}
Let $\mcal{Y}$ be a DM-stack. Then the category of morphisms $\mcal{Y} \ra \mcal{C}_{\mathbf{\De}}$
is equivalent to the category of $\bD$-collections over   $\mcal{Y}$.
\end{corollary}

\medskip

Before to proceed with the proof of the theorem, let me remark that
the stacks $\cC_{\mathbf{\De}}$ defined in this paper
are exactly the ``smooth toric DM stacks'' introduced in \cite{FMN} 
(see the appendix, Thm. \ref{relfmn}). 

We collect below some general results we need
in order to prove the above Theorem.

\begin{notation}\nf Following the notations used in \cite{B&Al}, 
for any quasi-affine group scheme $G$, we denote by
$$
p_{G} : BG \ra {\rm (Sch)}
$$
the structure morphism of the stack $BG$, and by
$$
\pi_{G}: {\rm Spec}\IC \ra BG
$$
the covering. We denote by
$$
b_G : {\rm pre}BG \ra BG
$$
the stackification morphism from the pre-stack ${\rm pre}BG$ to $BG$.
For any morphism $\varphi :G \ra H$ of group schemes, we denote by
$$
{\rm pre} B\varphi: {\rm pre}BG \ra {\rm pre}BH
$$
the induced morphism of pre-stacks.
\end{notation}

We quote a remark from \cite{B&Al}.
\begin{remark}\nf Let $\varphi: G\ra H$ be a morphism of
quasi-affine group schemes. Let $P\ra Y$ be a principal $G$-bundle.
Consider the right $G$-action on $P\times H$:
\begin{equation}\label{azione}
(x,h)\cd g = (x\cd g, \varphi (g^{-1}) \cd h).
\end{equation}
The theory of descent guarantees that there exists a quotient
scheme for the previous action,  $(P\times H)/G$, together with a morphism
$(P\times H)/G \ra Y$ inducing a principal $H$-bundle structure. 
The scheme $(P\times H)/G$ will be
denoted either as $P\times_{\varphi}H$ or by $P\times_{G}H$. 
Notice that  $P\times_{\varphi}H$ is not unique and due to this ambiguity
the correspondence
$P \mapsto P\times_{\varphi}H$
defines a functor $B\varphi: BG \ra BH$ up to unique canonical $2$-isomorphism.
In the following this ambiguity will be understood.

An analogous notation will be used to denote the quotients $P\times_G X$
in the case where $X$ is a quasi-affine scheme with a right $G$-action.
\end{remark}

We now recall two well known results in order to be  self-contained as much as possible. 
For complex algebraic varieties
the reference is \cite{Se}. For a more general context we refer to \cite{Gir},
Ch. III, Proposition 3.2.1. 

\begin{lemma}\label{lemma1}
Let
$$
1 \ra G \xrightarrow{\varphi}  H  \xrightarrow{\psi}  K \ra 1
$$
be a short exact sequence of quasi-affine group schemes.
Then
\begin{equation}\label{l1}
(B\varphi, p_{G}) : BG \ra BH \left._{B\psi} \times_{\pi_{K}} \right. {\rm Spec}\IC
\end{equation}
is an isomorphism of stacks.
\end{lemma}

\begin{lemma}\label{lemma2}
Let $G$ be a quasi-affine group scheme and let $X$ be a quasi-affine scheme
with a right $G$-action. For any principal $G$-bundle $\pi : P\ra Y$,
there is a bijection 
\begin{equation}\label{sem}
\{ \mbox{sections of}\quad P\times_G X \ra Y  \} \leftrightarrow
\{ G\mbox{-equivariant morphisms} \quad t:P\ra X \}
\end{equation}
which is functorial with respect to base change.
\end{lemma}

\vspace{0.3cm}

\begin{proof} (Of Theorem \ref{mainth}).
We proceed as follows:
we first prove that if $\{  \rho \in \De (1) \}$
spans $N_{\IQ}$ then $\cC_{\bD}$ is isomorphic to
$\cX_{\bS}$ as CFG, where $\cX_{\bS}$ is the stack defined in the statement of 
the Theorem, now the result follows from the fact that  $\cX_{\bS}$ 
is a separated smooth DM-stack \cite{BCS};
we proceed afterwards to prove the statement in general.
 
Identify the lattice $N$ with $\IZ^d$ and enumerate the $1$-dimensional cones of $\De$
as $\rho_1,...,\rho_n$. Then the $a_{\rho}\in N$ correspond to $(a_{1 k},...,a_{d k})\in \IZ^d$,
$k\in \{ 1,...,n \}$. Let us define the  matrices
\[
B:= \left( \begin{array}{ccc}
a_{11} & ... & a_{1n} \\
\vdots & ... & \vdots \\
a_{d1} & ... & a_{dn}\\
b_{11} & ... & b_{1n} \\
\vdots & ... & \vdots \\
b_{R1} & ... & b_{Rn}
  \end{array} \right); \qquad
Q:= \left(
\begin{array}{ccc}
0 & ... & 0 \\
\vdots & ... & \vdots \\
0 & ... & 0\\
r_1 & ... & 0 \\
\vdots & ... & \vdots \\
0 & ... & r_R
  \end{array} \right),
\]
where $B\in {\rm Mat}((d+R)\times n,\IZ)$, $Q\in {\rm Mat}((d+R)\times R,\IZ)$.
We consider the exact sequence
$$
0 \ra \left( \IZ^{d+R} \right)^* \xrightarrow{[BQ]^*}  \left( \IZ^{n+R} \right)^* \ra 
{\rm coker}([BQ]^*) \ra 0,
$$
where $[BQ]\in {\rm Mat}((d+R)\times (n+R),\IZ)$,
and we apply the functor ${\rm Hom}_{\IZ}(\_, \IC^*)$, we get an
exact sequence of affine group schemes:
\begin{equation}\label{sests}
1 \ra G \xrightarrow{\varphi} (\IC^*)^n \times (\IC^*)^R \xrightarrow{\psi} (\IC^*)^d \times (\IC^*)^R \ra 1
\end{equation}
where
\begin{eqnarray}\label{psi}
&& \psi (\lam_1,...,\lam_n, \mu_1,...,\mu_R)\\ && \quad = (\lam_1^{a_{11}} \cd \cd \cd \lam_n^{a_{1n}},...,
\lam_1^{a_{d1}} \cd \cd \cd \lam_n^{a_{dn}}, \mu_1^{r_1}\cd \lam_1^{b_{11}} \cd \cd \cd \lam_n^{b_{1n}},
...,\mu_R^{r_R}\cd \lam_1^{b_{R1}} \cd \cd \cd \lam_n^{b_{Rn}} ).\nonumber
\end{eqnarray}

The matrix $Q$ defines a morphism $\IZ^R \ra \IZ^{n+R}$ which is a projective resolution
of $N\op_{i} \IZ/ r_i$, and $B$ defines a lifting $\IZ^{n}\ra \IZ^{d+R}$
of the morphism $\IZ^n \ra N \op_{i} \IZ/ r_i$, 
$e_i \mapsto (a_{\rho_i}, \overline{b_{1 \rho_i}},...,\overline{b_{R\rho_i}})   $,
where $e_i$ is the $i$-th element of the standard basis of $\IZ^n$.
Following \cite{BCS} we associate to this data a toric DM-stack $\cX_{\bS}:=[Z/G]$.

We now define a functor of CFGs over $({\rm Sch})$
\begin{equation}\label{iso}
F : \mcal{X}_{\bS} \ra \mcal{C}_{\bD}.
\end{equation}
Consider an object of $\cX_{\bS}(Y)$,
 \begin{displaymath}
   \xymatrix{P \ar[d]_{\pi} \ar[r]^{t}& Z \\ Y & }
\end{displaymath}
Set 
$$
L_k := B\varphi (P) \times_{(\IC^*)^n \times (\IC^*)^R} \IC,
$$
the associate line bundle with respect to the action
\begin{eqnarray*}
\IC \times ((\IC^*)^n \times (\IC^*)^R) &\ra & \IC \\
z\cd (\lam_1,...,\lam_n, \mu_1,...,\mu_R) &\mapsto & z\cd \lam_k, 
\end{eqnarray*}
where $\varphi$ is defined in \eqref{sests}, and $k\in \{1,...,n\}$.
In the same way, for any $i\in \{ 1,...,R\}$, set 
$$
M_i := B\varphi (P) \times_{(\IC^*)^n \times (\IC^*)^R} \IC,
$$
be the line bundle associated to the action
\begin{eqnarray*}
\IC \times ((\IC^*)^n \times (\IC^*)^R) &\ra & \IC \\
z\cd (\lam_1,...,\lam_n, \mu_1,...,\mu_R) &\mapsto & z\cd \mu_{i}. \nonumber
\end{eqnarray*}

We now define the sections of $L_1,...,L_n$.
Recall that $G$ acts on $\IC^n$ by means of the $(\IC^*)^n$-component
 of $\varphi$. Then
the composition of $t: P \ra Z$ with the inclusion $Z \ra \IC^n$
gives an equivariant morphism $\ti{t} : P \ra \IC^n$.
By Lemma \ref{lemma2} we get a section $u$ of the vector bundle
\begin{equation}\label{sec}
P \times_G \IC^n =L_1 \op ... \op L_n,
\end{equation}
then $u_k$ is the $k$-th component of $u$ with respect to 
the decomposition \eqref{sec}.
Condition (2) of Definition \ref{deltacoll} follows from the fact that $\tilde{t} (P)\subset Z$
and from the definition of $Z$ as in \cite{BCS}.

We next define the isomorphisms $c_m$ required in the definition
of $\bD$-collection. By applying the functor  $(B\varphi , p_{G})$
defined in Lemma \ref{lemma1},  we get
\begin{equation}\label{nls}
(B\varphi , p_{G})(P) =(B\varphi (P), \al , Y\times ((\IC^*)^d \times (\IC^*)^R))
\end{equation}
where 
$\al : B\psi(B\varphi (P)) \ra Y\times ((\IC^*)^d \times (\IC^*)^R)$
is an isomorphism of principal $((\IC^*)^d \times (\IC^*)^R)$-bundles.
Consider now the diagonal action of $((\IC^*)^d \times (\IC^*)^R)$
on $\IC^{d+R}$ with weights $1$. Then from \eqref{psi} it follows that 
the associated vector bundle
$$
B\psi(B\varphi (P))\times_{(\IC^*)^d \times (\IC^*)^R} \IC^{d+R}
$$
is canonically isomorphic to
\begin{equation}\label{cd}
\op_{l=1}^d \left( \ot_{k=1}^n L_k^{\ot a_{lk}} \right) 
\op_{i}^R \left( M^{\ot r_i} \ot_{k=1}^n L^{\ot b_{ik}}  \right).
\end{equation}
The isomorphism $\al$ in \eqref{nls}
gives an isomorphism between \eqref{cd} and the trivial vector bundle,
then we get  isomorphisms  $c_{e_{1}^*},...,c_{e_{d}^*}, d_1,...,d_R$,
where $e_{1}^*,...,e_d^*$ is the dual basis of the standard basis
$e_1,...,e_d$ of $\IZ^d$. The $c_m$'s
are now uniquely determined by condition (1) of Definition \ref{deltacoll}.

We have defined $F$ on objects. The definition on morphisms
and the verification of the fact that it is a functor
is straightforward.
Moreover, $F$ is an equivalence of categories. This follows
from Lemmas \ref{lemma1} and \ref{lemma2}, and from
the equivalence between the category of principal $\IC^*$-bundles
and the one of line bundles. 
The result now follows since $\mcal{X}_{\bS}$ is a 
separated smooth DM-stack whose
coarse moduli space is the toric variety  associated to
the fan $\De$ and the lattice $N$
(Proposition 3.2 and Proposition 3.7 of \cite{BCS}).

We next consider the case where $\{ \rho \in \De(1) \}$ does not span
$N_{\IQ}$. 
We follow the ideas used in the proof of  Theorem 1.1 in \cite{cox}. Set
$$
N':= {\rm Span}( \{ \rho \in \De(1) \}) \cap N.
$$
The fan $\De$ can be regarded as a fan in $N'_{\IQ}$, then set
$$
\mathbf{\De}' =\{ N', \De, \{ a_{\rho} \}, r_1,...,r_R, \{ b_{i\rho} \}\}.
$$
From the first part of the proof we have that $\mcal{C}_{\mathbf{\De}'}$ is a 
separated smooth DM-stack.

 $N/N'$ is torsion free, so we can find a subgroup $N''$ of $N$ such that 
$N=N' \op N''$. The projection $N \ra N'$
determines an inclusion $\iota : M' \ra M$ such that
$M=\iota(M') \op N'^{\perp}$.

Let now $(L_{\rho},u_{\rho},c_m,M_i,d_i)/Y$ be a $\mathbf{\De}$-collection.
For any $m\in N'^{\perp}$, we have $\lan m, a_{\rho}\ran =0$
for all $\rho$. Thus $c_m$ can be identified with an element in $H^0(Y, \mcal{O}_Y^*)$.
Under this identification, the application $N'^{\perp} \ra H^0(Y, \mcal{O}_Y^*)$, $m \mapsto c_m$,
is a group homomorphism, thus induces a morphism of schemes
$Y \ra {\rm Spec}(\IC [N'^{\perp}])$.
In this way we get a morphism
\begin{equation}\label{pn'}
\mcal{C}_{\mathbf{\De}} \ra {\rm Spec}(\IC [N'^{\perp}]).
\end{equation}
On the other hand there is a morphism
\begin{equation}\label{pd'}
\mcal{C}_{\mathbf{\De}} \ra \mcal{C}_{\mathbf{\De}'}
\end{equation}
which associates $(L_{\rho},u_{\rho},c_m,M_i,d_i)/Y$
to $(L_{\rho},u_{\rho},\{c_m \, | \, m \in \iota(M') \}, M_i, d_i)/Y$.
Then the morphism  $\mcal{C}_{\mathbf{\De}} \ra \mcal{C}_{\mathbf{\De}'} \times {\rm Spec}(\IC [N'^{\perp}])$
 whose components are \eqref{pn'} and \eqref{pd'}
is an isomorphism. The result now follows since 
$\mcal{C}_{\mathbf{\De}'} \times {\rm Spec}(\IC [N'^{\perp}])$
is a smooth separated DM-stack and its coarse moduli space
is the toric variety associated to $\De$ and $N$. \end{proof}

\vspace{0.5cm}

\section{Classification of toric DM-stacks}
In this Section we show that the stack $\cC_{\bD}$ can be viewed 
as a gerbe banded by a finite abelian group. Then we study the problem of 
whether two combinatorial data $\bD$ and $\bD'$ define isomorphic
banded gerbes. We give an answer in combinatorial terms.

Let $\mathbf{\De}$ be a combinatorial data as in the Introduction,  set
\\
$$
\mathbf{\De}_{\rm rig} := \{ N\, ,\, \De \, ,\,  a_{\rho} | \rho \in \De(1) \, \}.
$$
\\
Consider the line bundles $\mathfrak{V}_1,...,\mathfrak{V}_R$  
on $\mcal{C}_{\mathbf{\De}_{\rm rig}}$ defined as follows: for any $\mathbf{\De}_{\rm rig}$-collection
$(L_{\rho},u_{\rho},c_m)/Y$, set
$$
{\mathfrak{V}_i }(Y) := \ot_{\rho}L_{\rho}^{\ot b_{i \rho}}, \qquad i\in \{ 1,...,R \};
$$
for any morphism 
$(f,\ga_{\rho}):(L_{\rho}',u_{\rho}',c_m')/Y' \ra (L_{\rho},u_{\rho},c_m)/Y $, set
$$
{\mathfrak{V}_i }(f,\ga_{\rho}):= \ot_{\rho}\ga_{\rho}^{\ot b_{i \rho}}.
$$

Let us denote by $\sqrt[r_i]{\mathfrak{V}_i}$ the gerbe  of $r_i$-th roots of
$\mathfrak{V}_i$, for $i\in \{1,...,R\}$. Just to fix notation we recall its definition here
and refer to \cite{Gir} Ch. IV (2.5.8.1) and \cite{AGVgwdms} for the general definition
and for more details.
For a scheme $Y$ over $\cC_{\bD_{\rm rig}}$, an object of $\sqrt[r_i]{\mathfrak{V}_i}(Y)$
is a pair $(M,d)$, where $M$  is a line bundle on $Y$ and $d:M^{\ot r_i}\ot {\mathfrak{V}_i }(Y) \ra \uIC$
is an isomorphism. If $(M',d')$ is an object over $Y'\ra \cC_{\bD_{\rm rig}}$, then a morphism
$(M',d')\ra (M,d)$ over $(f,\ga_{\rho})$ is given by a morphism of line bundles $\de: M' \ra M$
such that the following  diagram is cartesian
\begin{equation*}
\begin{CD}
M' @>\de>> M \\
@VVV @VVV \\
Y' @>f>> Y
\end{CD}
\end{equation*}
and $d\circ \de^{\ot r_i}\ot_{\rho}\ga_{\rho}^{\ot b_{i \rho}}  = d'$.

\vspace{0.5cm} 

Consider now the morphism of stacks
\begin{equation}\label{R}
\cR: \mcal{C}_{\mathbf{\De}} \ra \mcal{C}_{\mathbf{\De}_{\rm rig}}
\end{equation}
which associates to the $\mathbf{\De}$-collection $(L_{\rho}, u_{\rho}, c_m, M_i, d_i)/Y$
the $\mathbf{\De}_{\rm rig}$-collection $(L_{\rho}, u_{\rho}, c_m)/Y$,
and to $(f, \ga_{\rho}, \de_i)$ the arrow $(f,\ga_{\rho})$.
Then we have the following.\footnote{For a different proof of  
Prop. \ref{ger} see \cite{JT} and \cite{FMN}.} 

\vspace{0.5cm}

\begin{proposition}\label{ger}
The morphism \eqref{R} is a gerbe banded by
$\times_{i=1}^R \mu_{r_i}$ isomorphic to
\begin{equation}\label{gd}
\sqrt[r_1]{\mathfrak{V}_1} \times_{\mcal{C}_{\mathbf{\De}_{\rm rig}}}
\times ... \times_{\mcal{C}_{\mathbf{\De}_{\rm rig}}}
\sqrt[r_R]{\mathfrak{V}_R}.
\end{equation}
\end{proposition}

\medskip

\begin{proof} It is straightforward to verify that  $\cR$ is a gerbe.
Let now  $(L_{\rho},u_{\rho},c_m,M_i,d_i)/Y$ be a  $\mathbf{\De}$-collection.
The set of its automorphisms over the identity of $(L_{\rho},u_{\rho},c_m)/Y$ is
$$
\{ (\de_1,...,\de_R) \in H^0(Y, {\mcal{O}^*_Y})^R \, | \quad \de_i^{r_i}=1
\quad \mbox{for any} \quad i\in \{1,...,R\} \},
$$
then the canonical inclusions $\mu_{r_i}\subset \IC^*$, for $i\in \{ 1,...,R \}$, give the 
$\mu_{r_1}\times ... \times \mu_{r_R}$-banding.

Isomorphism classes of gerbes banded by
$\mu_{r_1}\times ... \times \mu_{r_R}$ are in 1-to-1 correspondence
with the second cohomology group $H^2(\mcal{C}_{\mathbf{\De}_{\rm rig}}, \times_{i=1}^R \mu_{r_i})$
\cite{Gir}.
There is an isomorphism
\begin{equation}\label{IV3.1.9}
H^2(\mcal{C}_{\mathbf{\De}_{\rm rig}}, \times_{i=1}^R \mu_{r_i}) \ra
\times_{i=1}^R H^2(\mcal{C}_{\mathbf{\De}_{\rm rig}},
\mu_{r_i})
\end{equation}
whose inverse associates the classes of the gerbes
$\mcal{G}_1,...,\mcal{G}_R$ to the class of the fibered product
$\mcal{G}_1 \times_{\mcal{C}_{\mathbf{\De}_{\rm
      rig}}}...\times_{\mcal{C}_{\mathbf{\De}_{\rm rig}}}\mcal{G}_R$.
The explicit description of 
\eqref{IV3.1.9} (as given for example in \cite{Gir} IV Proposition 2.3.18) shows that the
image of the class of $\cR$ is the class of \eqref{gd}.
This complete the proof. \end{proof}

\vspace{0.5cm}

\begin{remark}\nf We can also see $\cR$ as the rigidification of ${\cC}_{\bD}$ 
with respect to the constant sheaf $\times_{i=1}^R \mu_{r_i}$ \cite{ACV}
(this accounts for the subscript $\rm rig$ in the label $\bD_{\rm rig}$).
\end{remark}

\newpage

\begin{theorem}\label{classificazione}
For any toric DM-stack $\mcal{X}$, there exists a combinatorial data
$\mathbf{\De}=\{N,\De,a_{\rho},r_i,b_{i\rho}\}$ which satisfies the condition
\begin{equation}\label{cfag}
r_1 | r_2 | ... | r_R
\end{equation}
 such that $\mcal{X} \cong \mcal{C}_{\mathbf{\De}}$ (here $r_i | r_{i+1}$ denotes that $r_i$
divides $r_{i+1}$). 

Furthermore, if 
$\mathbf{\De}'=\{N,\De,a_{\rho},r_i^{'},b_{i\rho}^{'}\}$ is another combinatorial
data which  satisfies
\eqref{cfag}, then $\mcal{C}_{\mathbf{\De}} \cong \mcal{C}_{\mathbf{\De}'}$
as banded gerbes if and only if
$$
R=R', \quad r_i=r_i^{'} \quad \mbox{for all} \quad i,
$$ 
and the class
$$
\left[ \sum_{\rho} (b_{i\rho}-b_{i\rho}')e_{\rho}^* \right] \in (\IZ^{\De(1)})^*/M
$$
is divisible by $r_i$, for any $i \in \{ 1,...,R \}$.
Here $M$ is embedded in $(\IZ^{\De(1)})^*$ by the dual of the morphism
$\IZ^{\De(1)} \ra N$, $e_{\rho}\mapsto a_{\rho}$.
\end{theorem}
\begin{proof} Given $\mcal{X}$, the existence of
$\mathbf{\De}$ satisfying \eqref{cfag} follows from 
Theorem \ref{mainth}, Proposition \ref{ger}, and the classification
of finite abelian groups.
The condition \eqref{cfag} determines the isomorphism class
of the generic automorphism group of $\mcal{C}_{\mathbf{\De}}$. Hence
$\mcal{C}_{\mathbf{\De}} \cong \mcal{C}_{\mathbf{\De}'}$ implies that
$R=R'$ and $r_i=r_i^{'}$ for all $i$.
The same argument in the proof of  Proposition \ref{ger} shows that 
$\mcal{C}_{\mathbf{\De}} \cong \mcal{C}_{\mathbf{\De}'}$
as $\times_i \mu_{r_i}$-gerbes if and only if
$\sqrt[r_i]{\mathfrak{V}_i} \cong \sqrt[r_i]{\mathfrak{V}_i^{'}}$
as $\mu_{r_i}$-gerbe
for any $i\in \{ 1,...,R \}$. This is equivalent to the fact  that
$\mathfrak{V}_i \ot {\mathfrak{V}_i^{'}}^{-1}$ is a $r_i$-th power of
some element in ${\rm Pic}(\mcal{C}_{\mathbf{\De}_{\rm rig}})$. Now
the result follows from the fact that ${\rm
Pic}(\mcal{C}_{\mathbf{\De}_{\rm rig}})$ has the following
presentation \cite{BH}, \cite{DH}:
$$
0\ra M \ra (\IZ^{\De(1)})^* \ra {\rm Pic}(\mcal{C}_{\mathbf{\De}_{\rm
    rig}}) \ra 0,
$$
where the morphism $M\ra (\IZ^{\De(1)})^*$ is the dual of $\IZ^{\De(1)}\ra N$, 
$e_{\rho}\mapsto a_{\rho}$. \end{proof}

\vspace{0.2cm}

\begin{remark}\nf Notice that the first part of Thm. \ref{classificazione} holds in 
the more general situation where $\cX$ is
a ``smooth toric DM stack'' in the sense of \cite{FMN} (see the Appendix).
\end{remark}

\section{The torus action}
Any toric variety $X$ contains an algebraic torus $T$ as open dense
subvariety such that the action of $T$ on itself by multiplication extends
to an action on $X$. In this Section we show that an analogous
property holds for a toric DM-stack
once replacing $T$ with a \textit{Picard stack} $\mcal{T}$. 
Picard stacks (originally called \textit{champs de Picard}) were defined 
in Expos\'e XVIII  \cite{SGA4};
we refer to this paper for the definition and further properties. 
We will denote by $T$ the torus ${\rm Spec}(\IC [M])$, then any morphism
$g: Y \ra T$ is identified with the corresponding group homomorphism
$M \ra H^0(Y,{\mcal{O}}^*_Y)$, $m \mapsto g_m:= g^*(\chi^m)$.

Let us consider
$$
\mcal{T}:= \sqrt[r_1]{\uIC} \times_T ... \times_T \sqrt[r_R]{\uIC},
$$
and introduce the map
\begin{equation}
\mfm: \mcal{T} \times_{({\rm Sch})} \mcal{T}  \ra  \mcal{T}
\end{equation}
defined on objects by
\begin{equation*}
\mfm \left( (g, N_i, e_i)/Y,(g', N_i ', e_i ')/Y \right) := (g \cd g', N_i \ot N_i ', e_i \ot e_i ')/Y, 
\end{equation*}
and on arrows by $\mfm (\e_i, \e_i '):= \e_i \ot \e_i '$. 
The \textit{associativity} is expressed in terms
of a natural transformation
\begin{equation*}
\si: \mfm\circ (\mfm \times {\rm id}_{\mcal{T}}) \Rightarrow \mfm \circ ({\rm id}_{\mcal{T}} \times \mfm),
\end{equation*}
and the \textit{commutativity} with a natural transformation
\begin{equation*}
\tau : \mfm \Rightarrow \mfm\circ C,
\end{equation*}
where $C$ is the functor that exchanges the factors.
Here we choose the standard natural transformations $\si$ and $\tau$.
Then $({\mcal{T}},\mfm,\si,\tau)$ is a Picard stack.
Let us denote with
$e$  the neutral element of $\mcal{T}$, which is unique up to unique isomorphism.

\bigskip

The action of a group on a stack has been defined
in \cite{Ro}. The extension to the case of a group stack is straightforward,
the only difference here is that the action must be compatible 
with the associativity of $\cT$ which is expressed in terms of $\si$.
We follow the approach of \cite{Ro}, for a different one 
we refer to \cite{FMN}.

The Picard stack $\mcal{T}$ acts on $\mcal{C}_{\mathbf{\De}}$. 
The action is given by the functor 
\begin{equation*}
\mfa : \mcal{C}_{\mathbf{\De}} \times_{\rm (Sch)} \mcal{T}  \ra  \mcal{C}_{\mathbf{\De}} 
\end{equation*}
which is defined on objects as
$$
\mfa \left( (L_{\rho},u_{\rho},c_m,M_i,d_i) ,
(g_m,N_i,e_i) \right)/Y := (L_{\rho},u_{\rho},c_m\cd g_m,M_i\ot N_i,d_i\ot e_i)/Y,
$$
and on arrows as
$$
\mfa \left( (f,\ga_{\rho},\de_i), \e_i \right) := (f,\ga_{\rho},\de_i \ot \e_i ),
$$
and natural transformations
$\al: \mfa \circ ( {\rm id}_{\mcal{C}_{\mathbf{\De}}}\times \mfm ) \Rightarrow \mfa \circ (\mfa
\times {\rm id}_{\mcal{T}})$ and $\be: {\rm id}_{\mcal{C}_{\mathbf{\De}}} \Rightarrow \mfa \circ 
({\rm id}_{\mcal{C}_{\mathbf{\De}}} \times e ) $ such that  
the following diagrams are $2$-commutative
(we put in each square (resp. triangle) the appropriate natural transformation):
$$
\xymatrix{
\cC \times \cT \times \cT \times \cT \ar[dd]^{\mfa \times {\rm id_{\cT}}\times {\rm id_{\cT}} } 
\ar[dr]^{{\rm id_{\cC}}\times \mfm \times {\rm id_{\cT}}} \ar[rr]^{{\rm id_{\cC}}\times {\rm id_{\cT}} \times \mfm} & &  
\cC \times \cT \times \cT \ar'[d]^{\mfa \times {\rm id_{\cT}}}[dd] \ar[dr]^{{\rm id_{\cC}}\times \mfm} \\
& \cC \times \cT \times \cT \ar[dd]^<<<<<<{\mfa \times {\rm id_{\cT}}} \ar[rr]^<<<<<<{{\rm id_{\cC}}\times \mfm} & & 
\cC \times \cT \ar[dd]^{\mfa} \\
\cC \times \cT \times \cT \ar[dr]^{\mfa \times {\rm id_{\cT}}} \ar'[r]^>>>>>>>{{\rm id_{\cC}}\times \mfm} [rr]
& & \cC \times \cT \ar[dr]^{\mfa} & \\
& \cC \times \cT \ar[rr]^{\mfa} & & \cC 
}
$$
$$
\xymatrix{
\cC \ar@/_/[ddr]^>>>>>{{\rm id}_{\cC}\times e} \ar@/^/[drr]^{{\rm id}_{\cC}\times e} 
\ar[dr]^{{{\rm id}_{\cC}\times e} \times e} \\
& \cC \times \cT \times \cT \ar[d]^{\mfa \times {\rm id}_{\cT}} \ar[r]^{{\rm id}_{\cC}\times \mfm} &
\cC \times \cT \ar[d]^{\mfa} \\
& \cC \times \cT \ar[r]^{\mfa} & \cC
}
$$
where $\cC$ denotes $\cC_{\bD}$.
There is a standard choice for $\al$ and $\be$ which satisfy these conditions,
we adopt this one.

\begin{proposition}\label{dt}
There is a morphism 
\begin{equation*}\label{ae}
\mcal{T}\ra \mcal{C}_{\mathbf{\De}}
\end{equation*}
whose image is open and dense with respect to  the small \'etale site.
The restriction of $\mfa$ to $\mcal{T}$ with respect to the above morphism
is isomorphic to $\mfm$. 
\end{proposition}
\begin{proof}
Let $Y\ra \mcal{T}$ be a morphism given by $g:Y \ra T$,
$N_i$ and $e_{i}: N_i^{\ot r_i} \ra \uIC$,
for $i\in \{ 1,...,R \}$. Set $L_{\rho}:= \uIC$ and $u_{\rho}:= 1$
for all $\rho$, $c_m := g_{m}$, $m\in M$.
Then $(L_{\rho},u_{\rho},c_m,N_i,e_{i})$ is a $\mathbf{\De}$-collection.
This defines the morphism  on objects, on arrows
it sends $(f,\e_i)$ to $(f, {\rm id}, \e_i)$.

The morphism  is open with respect to the small \'etale site.
Indeed, let $\pi: U\ra \mcal{C}_{\mathbf{\De}}$ be an \'etale cover.
It corresponds to a $\mathbf{\De}$-family $(L_{\rho},u_{\rho},c_m,M_i,d_{i})/U$.
Set $U':= \{ x\in U  | u_{\rho}(x)\not= 0, \forall \rho \}$.
The restriction of $(L_{\rho},u_{\rho},c_m,M_i,d_{i})/U$ to $U'$ gives
a morphism $U' \ra \mcal{T}$ such that the following diagram is
$2$-Cartesian
\begin{equation*}
\begin{CD}
U' @>>> U \\
@VVV @VVV \\
\mcal{T} @>>>  \mcal{C}_{\mathbf{\De}}.
\end{CD}
\end{equation*}

The compatibility between $\mfa$ and $\mfm$ follows directly from the definition.
This completes the proof. \end{proof}

\vspace{0.5cm}

We proceed by observing that using the definition of \cite{BCS} it is possible to give
another description of the torus action. 
Indeed let $\phi :G \ra (\IC^*)^n$ be the composition of 
$\varphi$ in \eqref{sests} with the projection to $(\IC^*)^n$. 
Recall that $\cX_{\bS}=[Z/G]$, where the $G$-action is induced by $\phi$.
Let us consider the Picard stack
$\cG: = [(\IC^*)^n/G]$ (see 1.4.11 of Expos\'e XVIII in \cite{SGA4}).
Let ${\mcal{G}}^{\rm pre}$ and ${\mcal{X}}_{\mathbf{\Si}}^{\rm pre}$
be the pre-stacks associated to the groupoids
$((\IC^*)^n\times G \rightrightarrows (\IC^*)^n )$ and
$(Z\times G \rightrightarrows Z)$ respectively. ${\mcal{G}}^{\rm pre}$ 
acts on ${\mcal{X}}_{\mathbf{\Si}}^{\rm pre}$ in the obvious way
and the stackification of this action gives an action of
$\mcal{G}$ on ${\mcal{X}}_{\mathbf{\Si}}$. We denote
by $\mfa_{\cX}$ this action.

We want to compare the torus actions previously defined 
on $\mcal{X}_{\mathbf{\Si}}$ and on $\mcal{C}_{\mathbf{\De}}$.
Note that the restriction of \eqref{iso} to $\mcal{G}$ gives an isomorphism 
\begin{equation}\label{isot}
F_{|\mcal{G}}: \mcal{G}\ra {\mcal{T}}.
\end{equation} 
Moreover there is a natural transformation
\begin{equation}\label{e}
\nu:\mfa \circ (F\times F_{|\cG}) \Rightarrow F\circ \mfa_{\cX}
\end{equation}
defined as follows.
Consider the diagram
\begin{equation}\label{pree}
\begin{CD}
\cX^{\rm pre} \times \cG^{\rm pre} @>>> \cC_{\bD} \times \cT \\
@V\mfa_{\cX^{\rm pre}}VV @VV\mfa V \\
\cX^{\rm pre} @>>> \cC_{\bD}
\end{CD}
\end{equation}
where each row is the composition of $(F\times F_{|\cG})$ ($F$ resp.)
with the corresponding stackification morphism. 
There is a canonical  natural transformation $\nu^{\rm pre}$ which makes 
\eqref{pree} $2$-commutative. Then we define $\nu$ in \eqref{e} to be 
the unique natural transformation induced by $\nu^{\rm pre}$.
Finally we have the following.
\begin{proposition}
The isomorphism \eqref{iso} together with $\nu$ is $F_{|{\mcal{G}}}$-equivariant.
\end{proposition}
\begin{proof} Following \cite{Ro}, we have to prove that the diagrams below are 
$2$-commutative with respect to the natural transformations  previously defined:
$$
\xymatrix{
\cX \times \cG \times \cG \ar[dd]^{\mfa \times {\rm id}_{\cG}}
\ar[dr]^{F\times F_| \times F_|} \ar[rr]^{{\rm id}_{\cX}\times \mfm} & &  
\cX \times \cG \ar'[d]^{\mfa }[dd] \ar[dr]^{F\times F_|} \\
& \cC \times \cT \times \cT \ar[dd]^<<<<<<{\mfa \times {\rm id}_{\cT}} \ar[rr]^<<<<<<{{\rm id}_{\cC}\times \mfm} & & 
\cC \times \cT \ar[dd]^{\mfa} \\
\cX \times \cG \ar[dr]^{F \times F_|} \ar'[r][rr]^{\mfa} & & \cX \ar[dr]^{F} & \\
& \cC \times \cT \ar[rr]^{\mfa} & & \cC 
}
$$
$$
\xymatrix{
\cX \times \cG \ar[dr]^{F\times F_|} \ar[rr]^{\mfa} & & \cX \ar[dr]^{F}\\
& \cC \times \cT \ar[ur] \ar[rr]^{\mfa} & & \cC \\
\cX \ar[uu]^{{\rm id}_{\cX} \times e} \ar@{-}[ur]^{{\rm id}_{\cX}} \ar[dr]^{F}\\
& \cC \ar[uu]^{{\rm id}_{\cC}\times e} \ar[uurr]^{{\rm id}_{\cC}}
}
$$
With abuse of notation,  we have denoted with the same $\mfm$ the two multiplications 
of the Picard stacks and with the same $\mfa$ the two actions, 
$\cC$ denotes $\cC_{\bD}$.
Notice that it is enough to prove the $2$-commutativity of the above diagrams 
with $\cX$ and $\cG$ being replaced by $\cX^{\rm pre}$ and $\cG^{\rm pre}$ 
respectively. A direct computation shows that with these replacements
the above diagrams are indeed $1$-commutative, hence the result follows. \end{proof}

\section{Morphisms between toric stacks} 
 In this Section we give a 
description of morphisms between
toric DM-stacks parallel to the one given in  Theorem 3.2 of \cite{cox}
in the context of toric varieties. 
We need some introductory notations.
Let $\cY:=\cC_{\bD'}$ be the toric DM-stack defined by $\bD' := \{ N',\De',a_{\rho}'\}$
(notice that here we set $R'=0$). We assume that the coarse moduli space 
$Y$ of $\cY$ is a complete variety. Let $\cX:=\cC_{\bD}$ be a toric DM-stack
such that the $1$-dimensional rays generate $N_{\IQ}$.
Let us fix the presentations
$\cY=[Z'/G']$ and $\cX=[Z/G]$ as in \cite{BCS}.

The Picard group of $\cY$ is isomorphic 
to the group of characters of $G'$ \cite{BH}, \cite{DH}.
The isomorphism associates to the character $\chi$
the isomorphism class $[{\mcal{L}}(\chi)]\in {\rm Pic}(\mcal{Y})$
of the trivial line bundle on $Z'$
with the $G'$-linearization given by $\chi$. We use this isomorphism
to identify the two groups.

Considered $\cY$, we recall that there are distinguished elements
$[\cL_{\rho}]\in {\rm Pic}(\mcal{Y})$, for any $\rho \in \De'(1)$.
Let $\phi':G' \ra (\IC^*)^{\De'(1)}$
be the $(\IC^*)^{\De'(1)}$-component of $\varphi'$  in \eqref{sests}, 
the components $(\phi')_{\rho}$ of $\phi'$
are characters of $G'$, then the $[\cL_{\rho}]$'s are the corresponding
isomorphism classes of line bundles.

The homogeneous coordinate ring of $\mcal{Y}$ is
defined in \cite{DH}, it is 
the polynomial ring $S^{\mcal{Y}}:=\IC[z_{\rho}]$
in the variables $z_{\rho}$, where $\rho \in  \De'(1)$.
It is endowed with a  ${\rm Pic}(\mcal{Y})$-grading:
the monomial $\prod_{\rho}z_{\rho}^{l_{\rho}}$ has degree 
$\prod_{\rho} [{\mcal{L_{\rho}}}]^{l_{\rho}} \in {\rm Pic}(\mcal{Y})$.
For any $\chi \in  {\rm Pic}(\mcal{Y})$, let us denote by $S^{\mcal{Y}}_{\chi}$
the subset of $S^{\mcal{Y}}$ consisting of homogeneous polynomials of degree $\chi$.
There is an isomorphism of complex vector spaces \cite{DH}
$$
H^0(\mcal{Y},{\mcal{L}}(\chi)) \cong S^{\mcal{Y}}_{\chi}
$$
which we use to identify them.
With these notations we can now state the following
\begin{theorem}\label{mappe polinomiali}
Let $P_{\rho}\in S^{\mcal{Y}}$ be homogeneous polynomials indexed by $\rho \in \De(1)$, 
and let $\chi_i \in {\rm Pic}(\mcal{Y})$ for $i\in \{ 1,...,R \}$ such that
\begin{enuma}
\item If $P_{\rho}\in S^{\mcal{Y}}_{\chi_{\rho}}$, then $\sum_{\rho} \chi_{\rho}\ot a_{\rho}=0$
in ${\rm Pic}(\mcal{Y})\ot_{\IZ} N$, and $\prod_{\rho} \chi^{b_{i\rho}}_{\rho}\cd\chi_i^{r_i}=1$ 
in ${\rm Pic}(\mcal{Y})$ for any $i$,
\item $(P_{\rho}(z))\in Z$ whenever $z\in Z'$.
\end{enuma}
Let $\tilde{f}(z):= (P_{\rho}(z))\in \IC^{\De(1)}$. Then there is 
a morphism $f:\mcal{Y}\ra \mcal{X}$ such that the diagram 
\begin{equation*}
\begin{CD}
Z' @>\ti{f}>> Z \\
@VVV @VVV \\
\mcal{Y} @>f>> \mcal{X}
\end{CD}
\end{equation*}
is $2$-commutative, where the vertical arrows are the quotient maps. Furthermore$:$
\begin{enumi}
\item Let $\chi_i \in {\rm Pic}(\mcal{Y})$ fixed for $i\in \{ 1,...,R \}$. 
Then two sets of polynomials $\{ P_{\rho} \}$ and $\{ P_{\rho}' \}$
determine $2$-isomorphic morphisms if and only if there exists 
$g\in G$
such that $P_{\rho}'=\varphi_{\rho}(g)P_{\rho}$ for any $\rho \in \De(1)$,
where $\varphi_{\rho}$ is the $\rho$-th component of $\varphi$ defined in \eqref{sests}.
\item All morphisms $f:\mcal{Y}\ra \mcal{X}$ arise in this way up to $2$-isomorphisms.
\end{enumi}
\end{theorem}
\begin{proof} Let us choose a representative ${\mcal{L}}(\chi)$
of the class $\chi \in {\rm Pic}(\mcal{Y})$ for any $\chi$. Note that there
are canonical isomorphisms ${\mcal{L}}(\chi_1) \ot {\mcal{L}}(\chi_2) \cong {\mcal{L}}(\chi_1 \cd \chi_2)$.

Let $\{ P_{\rho} | \rho \in \De(1)\}$ and $\{\chi_i | i \in \{1,...,R\}\}$ 
satisfying (a) and (b). Then
$\prod_{\rho}\chi_{\rho}^{\lan m,a_{\rho} \ran} = 1$, for any $m\in M$. 
As a consequence there are canonical isomorphisms
$$
c_m^{\rm can}: \ot_{\rho}{\mcal{L}}(\chi_{\rho})^{\ot \lan m,a_{\rho} \ran} \ra \uIC,
\qquad m\in M,
$$
and
$$
d_i^{\rm can} : \ot_{\rho}{\mcal{L}}(\chi_{\rho})^{\ot b_{i\rho}}\ot {\mcal{L}}(\chi_i)^{\ot r_i} \ra \uIC,
\qquad i\in \{1,...,R\}.
$$
It follows that $({\mcal{L}}(\chi_{\rho}),P_{\rho},c_m^{\rm can},{\mcal{L}}(\chi_i),d_i^{\rm can})/\cY$
is a $\mathbf{\De}$-collection, therefore it corresponds to a morphism $f:\mcal{Y}\ra\mcal{X}$.
The commutativity of the diagram follows easily.

Let now $\{ P_{\rho} \}$ and $\{ P_{\rho}' \}$ be two sets of polynomials defining $2$-isomorphic morphisms. 
Then there are isomorphisms 
$$
\ga_{\rho}: {\mcal{L}}(\chi_{\rho}) \ra {\mcal{L}}(\chi_{\rho})\quad \mb{and} \quad
\de_i : {\mcal{L}}(\chi_{i}) \ra {\mcal{L}}(\chi_{i}),
$$
such that
$$
\ga_{\rho}(P_{\rho})=P_{\rho}'\; , \quad  \ot_{\rho}\ga_{\rho}^{\ot \lan m,a_{\rho} \ran}={\rm id} \; , \quad
\ot_{\rho}\ga_{\rho}^{\ot b_{i\rho}}\ot \de_i^{\ot r_i}={\rm id}
$$
 for any $\rho$, $m$ and $i$.
The $\ga_{\rho}$'s and  $\de_i $'s are multiplications by non-zero complex numbers
which we denote by the same symbols. The previous conditions means that 
$\left( (\ga_{\rho})_{\rho},(\de_i)_{i}  \right)\in {\rm ker}(\psi)$, where $\psi$ is defined in \eqref{sests}. 
So there exists  $g\in G$ such that $\varphi (g) = \left( (\ga_{\rho})_{\rho},(\de_i)_{i}  \right)$.

To conclude the proof, let $f:\cY \ra \cX$ be a morphism, and let
 $({\mcal{L}}_{\rho},u_{\rho},c_m,M_i,d_i)/\cY$ be the corresponding $\mathbf{\De}$-collection
(Corollary \ref{mfs}).
Then there is a morphism
$$
({\mcal{L}}_{\rho},u_{\rho},c_m,M_i,d_i) \ra ({\mcal{L}}(\chi_{\rho}),P_{\rho},\ti{c}_m,{\mcal{L}}(\chi_i),\ti{d}_i)
$$
 for some $\chi_{\rho},\chi_i\in {\rm Pic}(\mcal{Y})$. Clearly
the $P_{\rho}$'s, $\chi_{\rho}$'s and $\chi_i$'s satisfy conditions (a) and (b). Let us now consider the automorphisms
$(c_m^{\rm can})^{-1}\circ \ti{c}_m $ and $(d_i^{\rm can})^{-1}\circ \ti{d}_i $ of 
$\ot_{\rho}{\mcal{L}}(\chi_{\rho})^{\ot \lan m,a_{\rho} \ran}$ and 
$\ot_{\rho}{\mcal{L}}(\chi_{\rho})^{b_{i\rho}}\ot {\mcal{L}}(\chi_i)^{r_i}$ respectively.
They correspond to an element \\
$\left( ((c_m^{\rm can})^{-1}\circ \ti{c}_m)_m,((d_i^{\rm can})^{-1}\circ \ti{d}_i)_i \right)\in
(\IC^*)^d \times (\IC^*)^R$. 
Let now $\left( (\ga_{\rho})_{\rho},(\de_i)_i \right)\in (\IC^*)^n \times (\IC^*)^R$ such that 
$\psi \left( (\ga_{\rho})_{\rho},(\de_i)_i \right) =
\left( ((c_m^{\rm can})^{-1}\circ \ti{c}_m)_m,((d_i^{\rm can})^{-1}\circ \ti{d}_i)_i \right)$.
Then \eqref{psi} implies that 
$$
(\ga_{\rho},\de_i): ({\mcal{L}}(\chi_{\rho}),P_{\rho},\ti{c}_m,{\mcal{L}}(\chi_i),\ti{d}_i)
\ra  ({\mcal{L}}(\chi_{\rho}),P_{\rho},{c_m^{\rm can}},{\mcal{L}}(\chi_i),d_i^{\rm can})
$$
is a morphism of $\bD$-collections. From Corollary \ref{mfs} it follows that the morphism
associated to $ ({\mcal{L}}(\chi_{\rho}),P_{\rho},c_m^{\rm can},{\mcal{L}}(\chi_i),d_i^{\rm can})$
is $2$-isomorphic to $f$. \end{proof}

\section{Relations with logarithmic geometry}
In \cite{Iwanari} toric stacks have been defined in a different way
than in  \cite{BCS}, in this Section we briefly compare the two approaches.
We work over the field of complex numbers $\IC$,
keeping in mind that the construction of \cite{Iwanari} works
over a general field. Over a general field
toric stacks in the sense of \cite{Iwanari} are Artin stacks
of finite type with finite diagonal.

Let $\bD = \{ N, \De, a_{\rho} \}$ be a combinatorial data as in the Introduction,
note that here we assume $R=0$. In \cite{Iwanari}, for any scheme $Y$,
 the author considers
triples $(\pi: \cS \ra \cO_Y, \al :\cM \ra \cO_Y, \eta:\cS \ra \cM)$, where
\begin{enumerate}
\item $\cS$ is an \'etale sheaf of submonoids of the constant sheaf $M$
on $Y$ such that for any point $y\in Y$ the Zariski stalk $\cS_y$
is isomorphic to the \'etale stalk $\cS_{\bar y}$,
\item $\pi: \cS \ra \cO_Y $ is a morphism of monoids, where $\cO_Y$
is  a monoid with respect to the multiplication,
\item for any $s\in \cS$, $\pi(s)$ is invertible if and only if 
$s$ is invertible,
\item for any point $y\in Y$, there exists some $\si \in \De$
such that $\cS_{\bar y} = \si^{\vee}\cap M$,
\item $\al :\cM \ra \cO_Y$ is a fine logarithmic structure on $Y$ in the sense of \cite{Kato},
\item $\eta:\cS \ra \cM$ is an homomorphism of sheaves
of monoids such that $\pi = \al \circ \eta$, and for each generic point
$\bar{y}$, $\bar{\eta}: \left( \cS/\cO ^*_Y  \right)_{\bar{y}} \ra \left( \cM/\cO ^*_Y  \right)_{\bar{y}}$
is isomorphic to the $\bD$-\textit{free resolution} at $\bar{y}$.
\end{enumerate}
The notions of \textit{minimal free resolution} and of $\bD$-\textit{free resolution}
are the most important notions in the definition of \cite{Iwanari},
we refer to \cite{Iwanari} for more details.\\
One can define morphisms between two such triples.
Let $\cL_{\bD}$ be the resulting category. The  functor
\begin{equation}\label{log}
\cL_{\bD} \ra \rm{(Sch)}
\end{equation}
which forget the data $\pi: \cS \ra \cO_Y$, $\al :\cM \ra \cO_Y$
and $ \eta:\cS \ra \cM$ is a CFG. \\
Then, Theorem 2.7  in \cite{Iwanari} states that \eqref{log} is a smooth
DM-stack of finite type and separated, with coarse 
moduli space the toric variety associated to $N$ and $\De$.

\vspace{0.5cm}

Let $\bS$ be the stacky fan in the sense of \cite{BCS}
associated to $\bD$, and let $\cX_{\bS}$ be the 
toric stack as defined in \cite{BCS}.
The Proposition below follows from  our Theorem \ref{mainth}
and Theorem 1.4 in \cite{Iwanari2}.
\begin{proposition}
Under the previous hypothesis, there is an isomorphism of stacks
$$
\cL_{\bD} \cong \cX_{\bS}.
$$
\end{proposition}

\section{Appendix}
This appendix is aimed to give a
correspondence between the stacks $\cC_{\bD}$ defined in this paper 
and the ``smooth toric DM stacks''  defined in \cite{FMN}.
Let us recall from \cite{FMN} the following 
\begin{definition}
A ``smooth toric DM stack'' is a smooth separated DM-stack $\cX$
together with an open immersion of a Deligne-Mumford torus
$\cT \ra \cX$ with dense image such that the action of $\cT$
on itself extends to an action on $\cX$.
\end{definition}

\vspace{0.1cm}

We have the following result that can be seen as a generalization
of the functor of a smooth toric variety \cite{cox}.
\begin{theorem}\label{relfmn}
For any combinatorial data $\bD$, $\cC_{\bD}$ is a ``smooth toric DM stack''.
Conversely, for any ``smooth toric DM stack'' $\cX$ there exists a
combinatorial data  $\bD$
and an isomorphism $\cX \cong \cC_{\bD}$.
\end{theorem}
\begin{proof}
$\cC_{\bD}$ is a smooth and separated DM-stack (Thm. \ref{mainth}),
moreover there is an open immersion of a Picard stack
$\cT \ra \cC_{\bD}$ with dense image and such that the multiplication 
of $\cT$ extends to an action on $\cC_{\bD}$ (Prop. \ref{dt}).
As our $\cT$ is infact a Deligne-Mumford torus, the first part of the theorem
follows.

We divide the proof of the second part in three steps.\\
Step 1. Let $X$ be the coarse moduli space of $\cX$. Since $X$
is a simplicial toric variety, it is determined by a free abelian group 
$N$ and a fan $\De \subset N_{\IQ}$.  Moreover, $X$ has quotient singularities
by finite groups,
hence  it is the coarse moduli space of a ``canonical''
smooth DM-stack which we denote with $X^{\rm can}$ (we refer to \cite{FMN},
Def. 4.4, for the definition of ``canonical'' stack).  Let now $\cC_{\{N,\De, n_{\rho}\}}$ be the stack
associated to the combinatorial data $\{ N,\De, n_{\rho}\}$, 
where $n_{\rho}$ is the generator of the semigroup $\rho \cap N$ for any $\rho \in \De (1)$.
We have that $\cC_{\{N,\De, n_{\rho}\}} \cong X^{\rm can}$. 
Indeed, $\cC_{\{N,\De, n_{\rho}\}} \cong \cX_{\bS}\times (\IC^*)^k$,
where $\cX_{\bS}$ is a toric DM-stack associated to $\{ N,\De,n_{\rho} \}$ and 
$k$ is a natural number (Thm. \ref{mainth}); moreover $\cX_{\bS}\times (\IC^*)^k$
 is a ``canonical'' stack over $X$ (this can be seen using an orbifold atlas 
for $\cX_{\bS}$ as given for example by Prop. 4.3 in \cite{BCS}). 
The universal property of  ``canonical'' stacks implies that 
$X^{\rm can}\cong  \cC_{\{N,\De, n_{\rho}\}}$.  

\medskip

\noindent Step 2. Let now $\cX_{\rm rig}$ be the rigidification of $\cX$
by the generic stabilizer. Then, $\cX_{\rm rig}$ is a toric orbifold
 obtained from $X^{\rm can}$ by roots of effective Cartier divisors 
 (\cite{FMN} Thm. 5.2). More precisely, let $\rho_1,...,\rho_n $ be the $1$-dimensional cones
of $\De$, and 
let $D_{\rho_1},...,D_{\rho_n} \subset X^{\rm can}$ be the corresponding
torus invariant divisors. There are natural numbers $\al_{\rho_1},...,\al_{\rho_n}$
such that
$$
\cX_{\rm rig} \cong  \sqrt[\al_{\rho_1}]{D_{\rho_1}/X^{\rm can}}\times_{X^{\rm can}}...
\times_{X^{\rm can}}\sqrt[\al_{\rho_1}]{D_{\rho_1}/X^{\rm can}},
$$  
where $\sqrt[\al]{D/\cX}$ denotes the $\al$-th root of the effective Cartier
divisor $D$ in the smooth algebraic stack $\cX$. 
Set $a_{\rho}:=\al_{\rho}\cd n_{\rho}\in N$, for any $\rho \in \De(1)$. 
Then, under the identification of $X^{\rm can}$ with $\cC_{\{N,\De, n_{\rho}\}}$,
we obtain an isomorphism
$$
\cX_{\rm rig} \cong \cC_{\{N,\De, a_{\rho}\}}.
$$

\medskip

\noindent Step 3. $\cX$ is a gerbe over $\cX_{\rm rig}$ 
isomorphic to 
$$
  \sqrt[r_1]{\cL_1/\cX_{\rm rig}}\times_{\cX_{\rm rig}}...
\times_{\cX_{\rm rig}}\sqrt[r_R]{\cL_{R}/\cX_{\rm rig}},
$$
for some line bundles $\cL_1,...,\cL_R$ over $\cX_{\rm rig}$ and for some
natural numbers $r_1,...,r_R$ (\cite{FMN}, Cor. 6.27).
The Picard group of $\cX_{\rm rig}$ is generated by the 
classes of line bundles $\cO_{\cX_{\rm rig}}(\mathcal{D}_{\rho})$
associated to the torus invariant divisors $\mathcal{D}_{\rho}\subset \cX_{\rm rig}$
 for $\rho \in \De(1)$.
Hence, for any $i\in \{ 1,...,R\}$, there are integers
$b_{i\rho}$,  $\rho \in \De(1)$, with $\cL_i\cong \cO(\sum_{\rho}b_{i\rho}\mathcal{D}_{\rho})$.
It follows that 
$$
\cX \cong \cC_{\bD}
$$
with $\bD=\{N,\De,a_{\rho},b_{i\rho},r_1,...,r_R\}$.  
\end{proof}



\begin{thebibliography}{99}


\bibitem{ACV} 
\textsc{D. Abramovich, A. Corti and A. Vistoli},
Twisted bundles and admissible covers, 
Comm. Algebra 31 (2003),  3547--3618.

\bibitem{AGVgwdms} 
\textsc{D. Abramovich, T. Graber and A. Vistoli}, 
Gromov-Witten theory of Deligne-Mumford stacks, 
arXiv.org:math/0603151.

\bibitem{B&Al} 
\textsc{K. Behrend, B. Conrad, D. Edidin, W. Fulton, B. Fantechi, L. G\"ottsche and A. Kresch}, 
Algebraic Stacks, in preparation.

\bibitem{BCS} 
\textsc{L. A. Borisov, L. Chen and G. Smith}, 
The orbifold {C}how ring of toric {D}eligne-{M}umford stacks, 
J. Amer. Math. Soc. 18 (2005), 193--215.

\bibitem{BH} 
\textsc{L. A. Borisov and R. P. Horja}, 
On the {$K$}-theory of smooth toric {DM} stacks,  
Contemp. Math. 401 (2006), 21--42.

\bibitem{cox} 
\textsc{D. A. Cox}, 
The functor of a smooth toric variety,
Tohoku Math. J. 47 (1995), 251--262.


\bibitem{DH} 
\textsc{M. D\"urr and  M. Halic}, 
On toric stacks, preprint.

\bibitem{FMN}
\textsc{B. Fantechi, E. Mann and F. Nironi},
Smooth toric DM stacks,
 arXiv:0708.1254.

\bibitem{F} 
\textsc{W. Fulton}, 
Introduction to toric varieties,
Ann. of Math. Stud. 131, Princeton University Press, Princeton, 1993.

\bibitem{Gir} 
\textsc{J. Giraud}, 
Cohomologie non ab\'elienne,
Springer-Verlag, Berlin, 1971.

\bibitem{Iwanari} 
\textsc{I. Iwanari}, 
The category of toric stacks, arXiv:math/0610548.

\bibitem{Iwanari2} 
\textsc{I. Iwanari}, 
Generalization of Cox functor,  arXiv:0705.3524, version 1.

\bibitem{JT}
\textsc{Y. Jiang and H.-H. Tseng},
The integral (orbifold) {C}how ring of toric {D}eligne-{M}umford stacks,
 arXiv:0707.2972. 

\bibitem{Kato}
\textsc{K. Kato},
Logarithmic structures of Fontaine-Illusie,
in  Algebraic analysis, geometry, and number theory (Baltimore, MD, 1988), 191--224,
Johns Hopkins Univ. Press, Baltimore, MD, 1989. 

\bibitem{Ro} 
\textsc{M. Romagny}, 
Group actions on stacks and applications,
Michigan Math. J. 53 (2005), 209--236.

\bibitem{Se}
\textsc{J. P. Serre}, 
Espaces fibr\'es alg\'ebriques (d'apr\`es {A}ndr\'e {W}eil), 
S\'eminaire Bourbaki 2 (1995), 305--311.

\bibitem{SGA4} 
Th\'eorie des topos et cohomologie \'etale des sch\'emas.
{T}ome 3, (SGA 4), Lecture Notes in Mathematics 305, Springer-Verlag, Berlin, 1973.
\end{thebibliography}
\end{document}